\documentclass[a4paper,centertags,oneside,12pt]{amsart}
\usepackage{mathrsfs}
\usepackage{amssymb}
\usepackage{fancyhdr}
\usepackage{charter}
\usepackage{typearea}
\usepackage{pdfsync}
\usepackage{mathrsfs}
\newcommand{\comment}[1]{}

\usepackage[width=6.4in,height=8.5in]{geometry}

\newtheorem{theorem}{Theorem}[section]
\newtheorem{definition}[theorem]{Definition}
\newtheorem{proposition}[theorem]{Proposition}

\newtheorem{lemma}[theorem]{Lemma}

\newcommand{\cali}[1]{\mathscr{#1}}
\newcommand{\Cc}{\cali{C}}\newcommand{\cO}{\cali{O}}

\DeclareMathOperator{\codim}{codim}
\DeclareMathOperator{\Bs}{Bs}

\DeclareMathOperator{\ric}{Ric}
\newcommand{\MP}{{\rm MP}}
\newcommand{\dist}{\mathop{\mathrm{dist}}\nolimits}
\newcommand{\bfs}{{\rm \mathbf{s}}} 
\newcommand{\mI}{{\mathcal I}}
\newcommand{\mO}{{\mathcal O}}
\newcommand{\FS}{{\rm FS}}
\newcommand{\C}{\mathbb{C}}
\newcommand{\N}{\mathbb{N}}
\newcommand{\R}{\mathbb{R}}
\renewcommand\P{\mathbb{P}}

\begin{document}

\title[H\"older singular metrics on big line bundles and equidistribution]
{H\"older singular metrics on big line bundles and equidistribution}

\author{Dan Coman}
\thanks{D.\ Coman is partially supported by the NSF Grant DMS-1300157}
\address{Department of Mathematics, Syracuse University, Syracuse, NY 13244-1150, USA}\email{dcoman@syr.edu}
\author{George Marinescu}
\address{Universit{\"a}t zu K{\"o}ln, Mathematisches Institut, Weyertal 86-90, 50931 K{\"o}ln, 
Deutschland   \& Institute of Mathematics `Simion Stoilow', Romanian Academy, Bucharest, Romania}
\email{gmarines@math.uni-koeln.de}
\thanks{G.\ Marinescu is partially supported by DFG funded projects SFB/TR 12, MA 2469/2-2}
\author{Vi{\^e}t-Anh Nguy{\^e}n}
\address{Math{\'e}matique-B{\^a}timent 425, UMR 8628, Universit{\'e} Paris-Sud, 91405 Orsay, France}
\email{VietAnh.Nguyen@math.u-psud.fr}
\thanks{V.-A. Nguyen is partially supported by Humboldt  Foundation, Max-Planck Institute for Mathematics
and  Vietnam Institute for Advanced Study in Mathematics (VIASM)}
\thanks{Funded through the Institutional Strategy of the University of Cologne within the German Excellence Initiative}
\date{September 21, 2015}

\pagestyle{myheadings}

\begin{abstract}
We show that normalized currents of integration along the common zeros of random $m$-tuples of sections of 
powers of $m$ singular Hermitian big line bundles on a compact K\"ahler manifold distribute asymptotically 
to the wedge product of the curvature currents of the metrics. If the Hermitian metrics are H\"older with 
singularities we also estimate the speed of convergence.
\end{abstract}

\maketitle
\tableofcontents

\section{Introduction} \label{introduction}

\par Random polynomials or more generally holomorphic sections and the distribution of their zeros 
represent a classical subject in analysis
\cite{BlPo:32,ET50,Ham56,Kac49}, and they have been more recently used to model quantum
chaotic eigenfunctions \cite{BBL,NoVo:98}.

This area witnessed intense activity recently 
\cite{BL13,BS07,BMO14,DMS,DS06b,Sh08,ShZ99,ST04}, and especially results about
equidistribution of holomorphic sections in singular Hermitian holomorphic bundles were obtained
\cite{CM11,CM13,CM13b,CMM14,DMM14} with emphasis on the speed of convergence. 
The equidistribution is linked to the Quantum Unique Ergodicity conjecture of Rudnick-Sarnak \cite{RuSa:94}, 
cf.\  \cite{HolSou:10,Mars11}.

The equidistribution of common zeros of several sections is particularly interesting. Their study is difficult
in the singular context and equidistribution with the estimate of convergence speed
was established in \cite{DMM14} for H\"older continuous metrics.

In this paper we obtain the equidistribution of common zeros of sections of $m$ singular Hermitian
line bundles under the hypothesis that the metrics are continuous outside analytic sets
intersecting generically. We will moreover introduce the notion of H\"older metric with
singularities and establish the equidistribution with convergence speed of common zeros.
 
Let $(X,\omega)$ be a compact K\"ahler manifold of dimension $n$ and $\dist$ be 
the distance on $X$ induced  by $\omega$. If $(L,h)$ is a singular Hermitian holomorphic line bundle 
on $X$ we denote by $c_1(L,h)$ its curvature current. Recall that if $e_L$ is a holomorphic frame of $L$ 
on some open set $U\subset X$ then $|e_L|^2_h=e^{-2\phi}$, where $\phi\in L^1_{loc}(U)$ is called 
the {\it local weight} of the metric $h$ with respect to $e_L$, and $c_1(L,h)|_U=dd^c\phi$. 
Here $d=\partial+\overline\partial$,  $d^c= \frac{1}{2\pi i}(\partial -\overline\partial)$. 
We say that $h$ is {\it positively curved} if $c_1(L,h)\geq 0$ in the sense of currents. 
This is equivalent to saying that the local weights $\phi$ are plurisubharmonic (psh).

Recall that a holomorphic line bundle $L$ is called \emph{big} 
if its Kodaira-Iitaka dimension equals the dimension of $X$ (see
\cite[Definition\,2.2.5]{MM07}). By the Shiffman-Ji-Bonavero-Takayama 
criterion \cite[Lemma\,2.3.6]{MM07},
$L$ is big if and only if it admits a singular metric $h$ 
with $c_1(L,h)\geq\varepsilon\omega$ for some $\varepsilon>0$.

\par Let $(L_k,h_k)$, $1\leq k\leq m\leq n$, be $m$ singular Hermitian holomorphic line bundles on $(X,\omega)$. 
Let $H^0_{(2)}(X,L_k^p)$ be the Bergman space of $L^2$-holomorphic sections of 
$L_k^p:=L_k^{\otimes p}$ relative to the metric $h_{k,p}:=h_k^{\otimes p}$ 
induced by $h_k$ and the volume form $\omega^n$ on $X$, endowed with the inner product
\begin{equation}\label{e:ip}
(S,S')_{k,p}:=\int_{X}\langle S,S'\rangle_{h_{k,p}}\,\omega^n\,,\;\,S,S'\in H^0_{(2)}(X,L_k^p).
\end{equation}
Set $\|S\|_{k,p}^2=(S,S)_{k,p}$, $d_{k,p}=\dim H^0_{(2)}(X,L_k^p)-1$. For every $p\geq 1$ we consider the multi-projective space
\begin{equation}\label{e:X_p}
{\mathbb X}_p:= \P H^0_{(2)}(X,L^p_1)\times\ldots\times\P H^0_{(2)}(X,L^p_m)
\end{equation} 
equipped with the probability measure $\sigma_p$ which is the product of the Fubini-Study 
volumes on the components. If $S\in H^0(X,L_k^p)$ we denote by $[S=0]$ the current of 
integration (with multiplicities) over the analytic 
hypersurface $\{S=0\}$ of $X$.  Set 
\[
[\bfs_p=0]:=[s_{p1}=0]\wedge\ldots\wedge[s_{pm}=0]\,,\:\:\text{for 
$\bfs_p=(s_{p1},\ldots,s_{pm})\in {\mathbb X}_p$,}
\]
whenever this is well-defined (cf.\ Section \ref{S:FSB}). 
We also consider the probability space
$$(\Omega,\sigma_\infty):= \prod_{p=1}^\infty ({\mathbb X}_p,\sigma_p)\,.$$
Let us recall the following:
\begin{definition}\label{D:genpos}
We say that the analytic subsets $A_1,\ldots,A_m$, $m\leq n$, of a compact complex manifold $X$ of dimension $n$ are in general position if $\codim A_{i_1}\cap\ldots\cap A_{i_k}\geq k$ for every $1\leq k\leq m$ and $1\leq i_1<\ldots<i_k\leq m$.
\end{definition} 

Here is  our first main result.
\begin{theorem}\label{T:main1}
Let $(X,\omega)$ be a compact K\"ahler manifold of dimension $n$ and 
$(L_k,h_k)$, $1\leq k\leq m\leq n$, be $m$ singular Hermitian holomorphic line 
bundles on $X$ such that $h_k$ is continuous outside a proper analytic subset $\Sigma(h_k)\subset X$, 
$c_1(L_k,h_k)\geq\varepsilon\omega$ on $X$ for some $\varepsilon>0$, and $\Sigma(h_1),\ldots,\Sigma(h_m)$ 
are in general position. Then for $\sigma_\infty$-a.e. $\{\bfs_p\}_{p\geq1}\in \Omega$, 
we have in the weak sense of currents on $X$, 
$$\frac{1}{p^m}\,[\bfs_p=0]\to c_1(L_1,h_1)\wedge\ldots\wedge c_1(L_m,h_m)\,\text{ as }p\to\infty\,.$$ 
\end{theorem} 

\par In order to prove this theorem we show in Theorem \ref{T:speed} that the currents 
$\frac{1}{p^m}\,[\bfs_p=0]$ distribute as $p\to\infty$ like the wedge product of the normalized 
Fubini-Study currents of the spaces $H^0_{(2)}(X,L_k^p)$ defined in \eqref{e:FSdef} below. 
Then in Proposition \ref{P:FSwedge} we prove that the latter sequence of currents converges 
to $c_1(L_1,h_1)\wedge\ldots\wedge c_1(L_m,h_m)$.

\par Our second main result gives an estimate of the speed of convergence in 
Theorem \ref{T:main1} in the case when the metrics are H\"older with singularities.  

\begin{definition}\label{D:Hosing}
We say that a function $\phi:\ U\to[-\infty,\infty)$ defined on an open subset $U\subset X$ is 
H\"older with singularities along a proper analytic subset $\Sigma\subset X$ if there exist constants 
$c,\varrho>0$ and $0<\nu\leq 1$ such that  
\begin{equation}\label{e:Hdef}
|\phi(z)-\phi(w)|\leq\frac{c\,\dist(z,w)^\nu}{\min \{\dist(z,\Sigma),  \dist(w,\Sigma)\} ^\varrho}
\end{equation}
holds for all $ z,w\in U\setminus \Sigma$. A singular metric $h$ on $L$ is called H\"older with 
singularities along a proper analytic subset $\Sigma\subset X$ if all its local weights are 
H\"older functions with singularities along $\Sigma$.
\end{definition}

\par H\"older singular Hermitian metrics appear frequently in complex geometry
and pluri-potential theory.
Let us first observe that metrics with analytic singularities \cite[Definition\,2.3.9]{MM07},
which are very important
for the regularization of currents and for transcendental methods in algebraic geometry
\cite{BeDe12,CM13b,D90,D92,D93b},
are H\"older metrics with singularities. 
The  class of H\"older metrics with singularities  is invariant under pull-back and push-forward 
by meromorphic maps. In particular, this class is invariant under 
birational maps, e.\,g.\ blow-up and blow-down. They occur also as certain 
quasiplurisubharmonic upper envelopes (e.\,g.\  Hermitian metrics with minimal singularities 
on a big line bundle, equilibrium metrics, see \cite{BeBo10,DMM14,DMN15}, 
especially \cite[Theorem\,1.4]{BeDe12}).
   
\begin{theorem}\label{T:main2}   
In the setting of Theorem \ref{T:main1} assume in addition that $h_k$ is H\"older with singularities along $\Sigma(h_k)$. Then there exist a constant $\xi>0$ depending only on $m$ and a constant $c=c(X,L_1,h_1,\ldots,L_m,h_m)>0$ with the following property: For any sequence of positive numbers $\{\lambda_p\}_{p\geq1}$ such that 
$$\liminf_{p\to\infty} \frac{\lambda_p}{\log p}>(1+\xi n)c,$$ 
there are subsets $E_{p}\subset{\mathbb X}_p$ such that for $p$ large enough, 
\\[2pt]
(a) $\sigma_p(E_{p})\leq  c p^{\xi n} \exp(-\lambda_p/ c)$\,,
\\[2pt]
(b) if $\bfs_p\in {\mathbb X}_p\setminus E_{p}$ we have 
$$\Big|\Big \langle\frac{1}{p^m}[\bfs_p=0]-\bigwedge_{k=1}^mc_1(L_k,h_k),\phi\Big\rangle\Big|
\leq \frac{c\lambda_p}{p}  \| \phi\|_{\Cc^2}\,,$$ 
for any form $\phi$ of class $\Cc^2$\,.  

In particular, the last estimate holds for $\sigma_\infty$-a.e. sequence $\{\bfs_p\}_{p\geq1}\in \Omega$ provided that $p$ is large enough.
\end{theorem}

\par   Let $P_p$ be the Bergman kernel function 
of the space $H^0_{(2)}(X,L^p)$ defined in \eqref{e:Bk} below. The proof of Theorem \ref{T:main2} 
uses the estimate for $P_p$ obtained in Theorem \ref{T:Bka} in the case when the metric $h$ on 
$L$ is H\"older with singularities. 

\par One can readily specialize Theorem \ref{T:main2} to study the asymptotics with speed 
of common zeros of random $m$-tuples of sections of a (single) big line bundle endowed with a 
H\"older Hermitian metric with  singularities along a proper analytic subset $\Sigma\subset X$ 
of codimension $\geq m.$ Let $(L,h)$ be a singular Hermitian holomorphic line bundle on 
$(X,\omega)$ and $H^0_{(2)}(X,L^p)$ be the corresponding spaces of $L^2$-holomorphic sections. 
Consider the multi-projective space
$$ {\mathbb X}'_p:= (\P H^0_{(2)}(X,L^p))^m$$ 
endowed with the product probability measure $\sigma'_p$ induced by the Fubini-Study volume on 
$\P H^0_{(2)}(X,L^p)$, and let 
$$(\Omega',\sigma'_\infty):= \prod_{p=1}^\infty ({\mathbb X}'_p,\sigma'_p)\,.$$ 
If $\bfs_p=(s_{p1},\ldots,s_{pm})\in {\mathbb X}'_p$ we set 
$[\bfs_p=0]:=[s_{p1}=0]\wedge\ldots\wedge[s_{pm}=0]$, provided this current is well-defined. 
Applying Theorem \ref{T:main2} with $(L_k,h_k)=(L,h)$, $1\leq k\leq m$, 
and for the sequence $\lambda_p=(2+\xi n)c\log p$, we obtain:

\begin{theorem}\label{T:main3}
Let $(X,\omega)$ be a compact K\"ahler manifold of dimension $n$ and $(L,h)$ be a singular 
Hermitian holomorphic line bundle on $X$ such that $h$ is H\"older with  singularities along 
a proper analytic subset $\Sigma\subset X$ of codimension $\geq m,$ and 
$c_1(L,h)\geq\varepsilon\omega$ for some $\varepsilon>0$. 
Then there exist a constant $C>0$ depending only on $(X,\omega,L,h)$, 
and subsets $E_{p}\subset{\mathbb X}'_p\,$, such that for $p$ large enough,
\\[2pt]
(a) $\sigma'_p(E_{p})\leq Cp^{-2}$\,,
\\[2pt]
(b) if $\bfs_p\in {\mathbb X}'_p\setminus E_{p}$ we have 
$$\Big|\Big \langle\frac{1}{p^m}[\bfs_p=0]-c_1(L,h)^m,\phi\Big\rangle\Big|\leq 
C\,\frac{\log p}{p}\,\|\phi\|_{\Cc^2}\,,$$ 
for any form $\phi$ of class $\Cc^2$\,. 

In particular, the last estimate holds for $\sigma'_\infty$-a.e. sequence 
$\{\bfs_p\}_{p\geq1}\in \Omega'$ provided that $p$ is large enough.
\end{theorem}

This paper is organized as follows. In Section \ref{S:Bka} we prove a pointwise estimate 
for the Bergman kernel function in the case of H\"older metrics with singularities. 
Section \ref{S:FSB} is devoted to the study of the intersection of Fubini-Study currents 
and to a version of the Bertini theorem. In Section \ref{S:MT} we consider the Kodaira map 
as a meromorphic transform and estimate the speed of convergence of the intersection 
of zero-divisors of $m$ bundles. We use this to prove Theorem \ref{T:main1}. 
Finally, in Section \ref{S:PT:main2&3} we prove Theorem \ref{T:main2}. 

\medskip

{\it Acknowledgment.} We thank Vincent Guedj for useful discussions regarding this paper. 
We are grateful to the referee for a nice idea to improve 
an earlier version of  Theorem \ref{T:FSspeed1}.  
 

\section{Asymptotic behavior of Bergman kernel functions}\label{S:Bka}

\par In this section we prove a theorem about the asymptotic behavior of the Bergman 
kernel function in the case  when the metric  is H\"older with  singularities. 

\par Let $(L,h)$ be a holomorphic line  bundle over a compact K\"ahler manifold $(X,\omega)$ 
of dimension $n$, where $h$ is a singular Hermitian metric on $L$. Consider the space 
$H^0_{(2)}(X,L^p)$ of $L^2$-holomorphic sections of $L^p$ relative to the metric 
$h^p:=h^{\otimes p}$ induced by $h$ and the volume form  
$\omega^n$ on $X$, endowed with the natural inner product (see \eqref{e:ip}). 
Since $H^0_{(2)}(X,L^p)$ is finite dimensional, let $\{S^p_j\}_{j=0}^{d_p}$ 
be an orthonormal basis and denote by $P_p$ the Bergman kernel function defined by 
\begin{equation}\label{e:Bk}
P_p(x)=\sum_{j=0}^{d_p}|S^p_j(x)|_{h^p}^2,\;\;|S^p_j(x)|_{h^p}^2:=
\langle S_j^p(x),S_j^p(x)\rangle_{h^p},\;x\in X.
\end{equation}
Note that this definition is independent of the choice of basis.

\begin{theorem}\label{T:Bka} 
Let $(X,\omega)$ be a compact K\"ahler manifold of dimension $n$ and $(L,h)$ be a singular 
Hermitian holomorphic line bundle on $X$ such that $c_1(L,h)\geq\varepsilon\omega$ for some 
$\varepsilon>0$. Assume that $h$ is H\"older with singularities along a proper analytic subset 
$\Sigma$ of $X$ and with parameters $\nu,\,\varrho$ as in \eqref{e:Hdef}.
 If $P_p$ is the Bergman kernel function defined by \eqref{e:Bk} for the space 
 $H^0_{(2)}(X,L^p),$ then there exist a constant $c>1$ and $p_0\in\N$ 
 which depend only on $(X,\omega,L,h)$ such that for all 
 $z\in X\setminus\Sigma$ and all $p\geq p_0$
\begin{equation}\label{e:Bka}
\frac{1}{c}\leq P_p(z) \leq  \frac{cp^{2n/\nu}}{\dist (z,\Sigma)^{2n\varrho/\nu}}\;\cdot
\end{equation}
\end{theorem}

\par Recall that by Theorem 5.3 in \cite{CM11} we have 
$\lim_{p\to\infty}\frac{1}{p}\,{ \log P_p(z)}=0 $ locally uniformly on 
$X\setminus \Sigma$ for any metric $h$   which is only continuous outside of $\Sigma$. 
Theorem \ref{T:Bka} refines \cite[Theorem 5.3]{CM11} in this context, and it is interesting
to compare it to the asymptotic expansion of the Bergman kernel function in the case of smooth metrics \cite{Be03,Ca99,HM12,MM07,MM08,Ti90,Z98}.

\begin{proof} 
The proof follows from \cite[Section\,5]{CM11}, which is based on techniques of 
Demailly \cite[Proposition\,3.1]{D82}, \cite[Section\,9]{D93b}.
Let $x\in X$ and $U_\alpha\subset X$ be a coordinate neighborhood of $x$ 
on which there exists a holomorphic frame $e_\alpha$ of $L$. Let $\psi_\alpha$ 
be a psh weight of $h$ on $U_\alpha$. Fix $r_0>0$ so that the (closed) ball 
$V:=B(x,2r_0)\Subset U_\alpha$ and let $U:=B(x,r_0)$. 
By \cite[(7)]{CM11} there exist constants $c_1>0$, $p_0\in\mathbb{N}$ so that 
$$-\frac{\log c_1}{p}\leq\frac{1}{p}\,\log P_p(z)\leq
\frac{\log(c_1r^{-2n})}{p}+2\left(\max_{B(z,r)}\psi_\alpha-\psi_\alpha(z)\right)$$
holds for all $p>p_0$, $0<r<r_0$ and $z\in U$ with $\psi_\alpha(z)>-\infty$.  

\par For $z\in U\setminus\Sigma$ and $r<\min\{\dist(z,\Sigma),r_0\}$ we have since $\psi_\alpha$ is H\"older that 
 $$\max_{B(z,r)}\psi_\alpha-\psi_\alpha(z)\leq\frac{cr^{\nu}}{\left(\dist(z,\Sigma)-r\right)^\varrho}\,,$$
where $c>0$ depends only on $x$. Taking $r=\dist(z,\Sigma)^{\varrho/\nu}p^{-1/\nu}<\dist(z,\Sigma)/2$ (for $p_0$ large enough), we obtain 
\begin{eqnarray*}
-\log c_1\leq\log P_p(z)&\leq&\log c_1-2n\log r+2^{\varrho+1}cpr^{\nu}\dist(z,\Sigma)^{-\varrho}\\
&=&c_0+2n\log\left(\dist(z,\Sigma)^{-\varrho/\nu}p^{1/\nu}\right)\;.
\end{eqnarray*}
This holds for all $z\in U\setminus\Sigma$ and $p>p_0$, with constants $r_0,p_0,c_0,c_1$ 
depending only on $x$. A standard compactness argument now finishes the proof.
\end{proof}


\section{Intersection of Fubini-Study currents and Bertini type theorem}\label{S:FSB}

\par In this section we show that the intersection of the Fubini-Study currents associated with line bundles as in Theorem \ref{T:main1} is well-defined. Moreover, we show that the sequence of wedge products of normalized Fubini-Study currents converges weakly to the wedge product of the curvature currents of $(L_k,h_k)$. We then prove that almost all zero-divisors of sections of large powers of these bundles are in general position in the sense of Definition \ref{D:genpos}. 

\par Let $V$ be a vector space of complex dimension $d+1$. If $V$ is endowed with a Hermitian metric, 
then we denote by $\omega_{_\FS}$ the induced Fubini-Study form on the projective space $\P(V)$ 
(see \cite[pp.\,65,\,212]{MM07}) normalized so that $\omega_{_{\FS}}^d$ is a probability measure. 
We also use the same notations for $\P(V^*)$.

\par We keep the  hypotheses and notation of Theorem \ref{T:main1}. Namely, $(L_k,h_k)$, $1\leq k\leq m\leq n$, are singular Hermitian holomorphic line bundles on the compact K\"ahler manifold $(X,\omega)$ of dimension $n$, such that $h_k$ is continuous outside a proper analytic subset $\Sigma(h_k)\subset X$, $c_1(L_k,h_k)\geq\varepsilon\omega$ for some $\varepsilon>0$, and $\Sigma(h_1),\ldots,\Sigma(h_m)$ are in general position in the sense of Definition \ref{D:genpos}.

Consider the space $H^0_{(2)}(X,L_k^p)$ of $L^2$-holomorphic sections of $L^p_k$ endowed with the inner product \eqref{e:ip}. Since  $c_1(L_k,h_k)\geq\varepsilon \omega,$ it is  well-known that $H^0_{(2)}(X,L_k^p)$ is nontrivial for $p$ sufficiently large, see e.\,g.\ Proposition \ref{P:Bsdim}. Let 
\[
d_{k,p}:=\dim H^0_{(2)}(X,L_k^p)-1.
\]
The  {\it Kodaira map} associated  with $(L^p_k,h_{k,p})$  is  defined by
\begin{equation}\label{e:Kodaira_alg}
\Phi_{k,p}:\  X\dashrightarrow  {\mathbb G}(d_{k,p}, H^0_{(2)}(X,L_k^p))\,,\:\:
\Phi_{k,p}(x):=\left\lbrace s\in H^0_{(2)}(X,L_k^p) :\  s(x)=0  \right\rbrace,
\end{equation}
where  ${\mathbb G}(d_{k,p}, H^0_{(2)}(X,L_k^p)) $ denotes the Grassmannian 
of hyperplanes in $H^0_{(2)}(X,L_k^p)$  (see \cite[p.\,82]{MM07}). 
Let us identify ${\mathbb G}(d_{k,p}, H^0_{(2)}(X,L_k^p))$ with $\P(H^0_{(2)}(X,L_k^p)^*)$ by sending a hyperplane to an equivalence class of non-zero complex linear functionals  on $H^0_{(2)}(X,L_k^p)$ having the hyperplane as their common kernel. By  composing $\Phi_{k,p}$ with this identification, we  obtain a meromorphic map  
\begin{equation}\label{e:Kodaira_alg_dual}
\Phi_{k,p}:\  X\dashrightarrow \P(H^0_{(2)}(X,L_k^p)^*).
\end{equation}
To get an  analytic  description of $\Phi_{k,p},$
let 
\begin{equation}\label{e:basis}
S^{k,p}_j\in H^0_{(2)}(X,L_k^p),\,j=0,\ldots,d_{k,p}\,,
\end{equation} 
be an orthonormal basis and denote by $P_{k,p}$ the Bergman kernel function of the space 
$H^0_{(2)}(X,L_k^p)$ defined as in \eqref{e:Bk}.
This  basis gives identifications $H^0_{(2)}(X,L_k^p)\simeq \C^{d_{k,p}+1}$ and 
$\P(H^0_{(2)}(X,L_k^p)^*)\simeq \P^{d_{k,p}}$.
Let $U$ be a contractible Stein open set in $X$, let $e_k$ be a local holomorphic frame for $L_k$ on $U$, 
and write $S_j^{k,p}=s_j^{k,p}e_k^{\otimes p},$ where  $s_j^{k,p}$ is a  holomorphic  function on $U.$
 By  composing $\Phi_{k,p}$ given in (\ref{e:Kodaira_alg_dual}) with the last  identification, 
 we  obtain a meromorphic map     $\Phi_{k,p}:X\dashrightarrow\P^{d_{k,p}}$   
which has the  following local expression
\begin{equation}\label{e:Kodaira}
\Phi_{k,p}(x)=[s_0^{k,p}(x):\ldots:s_{d_{k,p}}^{k,p}(x)]\,\text{ for }x\in U.
\end{equation}
It is  called {\it the Kodaira map defined by the basis $\{S^{k,p}_j\}_{j=0}^{d_{k,p}}$}.  

\par Next, we define the \emph{Fubini-Study currents} $\gamma_{k,p}$ of $H^0_{(2)}(X,L_k^p)$ by 
\begin{equation}\label{e:FSdef}
\gamma_{k,p}|_U=\frac{1}{2}\,dd^c\log\sum_{j=0}^{d_{k,p}}|s_j^{k,p}|^2,\;   
\end{equation}
where the open set $U$ and the holomorphic functions $s_j^{k,p}$ are as  above. 
Note that $\gamma_{k,p}$ is a positive closed current of bidegree $(1,1)$ on $X$, 
and is independent of the choice of basis. Actually, the Fubini-Study currents are pullbacks 
of the Fubini-Study forms by Kodaira maps, which justifies their name.

\par Let $\omega_{_{\FS}}$ be  the Fubini-Study form  on $\P^{d_{k,p}}.$  
By \eqref{e:Kodaira} and \eqref{e:FSdef},
the currents $\gamma_{k,p}$ can be described as pullbacks 
\begin{equation}\label{e:FSG}
\gamma_{k,p}=\Phi_{k,p}^*(\omega_{_{\FS}}),\quad 1\leq k\leq m.
\end{equation}  
 We introduce the psh function 
\begin{equation}\label{e:FSpot}
u_{k,p}:=\frac{1}{2p}\log\sum_{j=0}^{d_{k,p}}|s_j^{k,p}|^2=u_k+\frac{1}{2p}\,\log P_{k,p}\,\;\text{ on }\,U\,,
\end{equation}
where $u_k$ is the weight of the metric $h_k$ on $U$ corresponding to $e_k$, so $|e_k|_{h_k}=e^{-u_k}$.
Clearly, by \eqref{e:FSdef} and \eqref{e:FSpot},   $dd^c u_{k,p}=\frac{1}{p}\,\gamma_{k,p}$. 
Moreover, note that by \eqref{e:FSpot}, $\log P_{k,p}\in L^1(X,\omega^n)$ and 
\begin{equation}\label{e:FSB}
\frac{1}{p}\,\gamma_{k,p}=c_1(L_k,h_k)+\frac{1}{2p}\,dd^c\log P_{k,p}
\end{equation} 
as currents on $X$. By \cite[Theorem 5.1, Theorem 5.3]{CM11} (see also \cite[(7)]{CM11}) there exist $c>0$, $p_0\in\N$, such that if $p\geq p_0$, $1\leq k\leq m$ and $z\in X\setminus\Sigma(h_k)$, then $P_{k,p}(z)\geq c$. By \eqref{e:FSpot} it follows that 
\begin{equation}\label{e:FScomp}
u_{k,p}(z)\geq u_k(z)+\frac{\log c}{2p}\;,\,\;z\in U,\;p\geq p_0,\;1\leq k\leq m.
\end{equation}
For  $p\geq 1$ consider the following analytic subsets of $X$: 
$$\Sigma_{k,p}:=\left\lbrace x\in X:\,S^{k,p}_j(x)=0,\;0\leq j\leq d_{k,p}\right\rbrace,\;1\leq k\leq m\,.$$
Hence $\Sigma_{k,p}$ is the base locus of $H^0_{(2)}(X,L_k^p)$, 
and $\Sigma_{k,p}\cap U=\{u_{k,p}=-\infty\}$. Note also that $\Sigma(h_k)\cap U\supset\{u_k=-\infty\}$ and by \eqref{e:FScomp} we have $\Sigma_{k,p}\subset\Sigma(h_k)$ for $p\geq p_0$. 
\begin{proposition}\label{P:FSwedge}
In the hypotheses of Theorem \ref{T:main1} we have the following:
\smallskip
\par
(i) For all $p$ sufficiently large and every $J\subset\{1,\ldots,m\}$ the analytic sets $\Sigma_{k,p}\,$, $k\in J$, $\Sigma(h_\ell)\,$, $\ell\in J':=\{1,\ldots,m\}\setminus J$, are in general position. 
\smallskip
\par
(ii) If $p$ is sufficiently large then the currents 
$$\bigwedge_{k\in J}\gamma_{k,p}\;\wedge\,\bigwedge_{\ell\in J'}c_1(L_\ell,h_\ell)$$
are well defined on $X$, for every $J\subset\{1,\ldots,m\}$.
\smallskip
\par
(iii) $\frac{1}{p^m}\,\gamma_{1,p}\wedge\ldots\wedge\gamma_{m,p}\to c_1(L_1,h_1)\wedge\ldots\wedge c_1(L_m,h_m)$ as $p\to\infty$, in the weak sense of currents on $X$.
\end{proposition}
\begin{proof} As noted above we have by \eqref{e:FScomp} that $\Sigma_{k,p}\subset\Sigma(h_k)$ for all $p$ sufficiently large. Since $\Sigma(h_1),\ldots,\Sigma(h_m)$ are in general position this implies $(i)$. Then $(ii)$ follows by \cite[Corollary 2.11]{D93}. 

\par $(iii)$ Let $U\subset X$ be a contractible Stein open set as above, $u_{k,p}$, $u_k$ 
be the psh functions defined in \eqref{e:FSpot}, so $dd^cu_k=c_1(L_k,h_k)$ and 
$dd^cu_{k,p}=\frac{1}{p}\,\gamma_{k,p}$ on $U$. By \cite[Theorem 5.1]{CM11} 
we have that $\frac{1}{p}\,\log P_{k,p}\to0$ in $L^1(X,\omega^n)$, 
hence by \eqref{e:FSpot}, $u_{k,p}\to u_k$ in $L^1_{loc}(U)$, as $p\to\infty$, 
for each $1\leq k\leq m$. Recall that by \eqref{e:FScomp}, 
$u_{k,p}\geq u_k-\frac{C}{p}$ holds on $U$ for all $p$ sufficiently large 
and some constant $C>0$. Then \cite[Theorem 3.5]{FS95} implies that 
$dd^cu_{1,p}\wedge\ldots\wedge dd^cu_{m,p}\to dd^cu_1\wedge\ldots\wedge dd^cu_m$ 
weakly on $U$ as $p\to\infty$. 
\end{proof}

\par We will need the following version of Bertini's theorem. The corresponding statement for the case of a single line bundle is proved in \cite[Proposition 4.1]{CM11}.

\begin{proposition}\label{P:Bertini} Let $L_k\longrightarrow X$, $1\leq k\leq m\leq n$, be holomorphic line bundles over a compact complex manifold $X$ of dimension $n$. Assume that: 
\smallskip
\par (i) $V_k$ is a vector subspace of $H^0(X,L_k)$ with basis $S_{k,0},\dots, S_{k,d_k}$, base locus $\Bs V_k:=\{S_{k,0}=\ldots=S_{k,d_k}=0\}\subset X$, such that $d_k\geq1$ and the analytic sets $\Bs V_1,\ldots,\Bs V_m$ are in general position in the sense of Definition \ref{D:genpos}.
\smallskip
\par (ii) $Z(t_k):=\{x\in X:\,\sum_{j=0}^{d_k}t_{k,j}S_{k,j}(x)=0\}$, where $t_k=[t_{k,0}:\ldots:t_{k,d_k}]\in\P^{d_k}$.
\smallskip
\par (iii) $\nu=\mu_1\times\ldots\times\mu_m$ is the product measure on $\P^{d_1}\times  \ldots\times\P^{d_m}$,  where $\mu_k$ is the Fubini-Study volume on $\P^{d_k}$. 
\smallskip
\par Then the analytic sets $Z(t_1),\ldots,Z(t_m)$ are in general position for $\nu$-a.e. $(t_1,\ldots,t_m)\in \P^{d_1}\times\ldots\times\P^{d_m}$. 
\end{proposition}

\begin{proof} If $1\leq l_1<\ldots<l_k\leq m$ let $\nu_{l_1\ldots l_k}=\mu_{l_1}\times\ldots\times\mu_{l_k}$ be the product measure on $\P^{d_{l_1}}\times\ldots\times\P^{d_{l_k}}$. For $1\leq k\leq m$ consider the sets
$$U_k=\{(t_{l_1},\ldots,t_{l_k})\in\P^{d_{l_1}}\times\ldots\times\P^{d_{l_k}}:\,\dim Z(t_{l_1})\cap\ldots\cap Z(t_{l_k})\cap A_j\leq n-k-j\}\,,$$
where $1\leq l_1<\ldots<l_k\leq m$, $j=0$ and $A_0=\emptyset$, or $1\leq j\leq m-k$ and $A_j=\Bs V_{i_1}\cap\ldots\cap\Bs V_{i_j}$ for some $i_1<\ldots<i_j$ in $\{1,\ldots,m\}\setminus\{l_1,\ldots,l_k\}$.

\par The proposition follows if we prove by induction on $k$ that 
$$\nu_{l_1\ldots l_k}(U_k)=1$$ 
for every set $U_k$ with $1\leq l_1<\ldots<l_k\leq m$, $0\leq j\leq m-k$ and $A_j$ as above. Clearly, it suffices to consider the case $\{l_1,\ldots,l_k\}=\{1,\ldots,k\}$. To simplify notation we set $\nu_k:=\nu_{1\ldots k}$. 

\par Let $k=1$. If $j=0$, $A_0=\emptyset$, so $U_1=\{t_1\in\P^{d_1}:\,\dim Z(t_1)\leq n-1\}=\P^{d_1}$. Assume next that $1\leq j\leq m-1$ and write $A_j=\bigcup_{l=1}^ND_l\cup B$, where $D_l$ are the irreducible components of $A_j$ of dimension $n-j$ and $\dim B\leq n-j-1$. We have that $\{t_1\in\P^{d_1}:\,D_l\subset Z(t_1)\}$ is a proper linear subspace of $\P^{d_1}$. Indeed, otherwise $D_l\subset\Bs V_1$, so $\dim A_j\cap\Bs V_1=n-j$, which contradicts the hypothesis that $\Bs V_1,\ldots,\Bs V_m$ are in general position. If $t_1\in\P^{d_1}\setminus U_1$ then $\dim Z(t_1)\cap A_j\geq n-j$. Since $Z(t_1)\cap A_j$ is an analytic subset of $A_j$, it follows that $D_l\subset Z(t_1)\cap A_j$ for some $l$, hence $\P^{d_1}\setminus U_1=\bigcup_{l=1}^N\{t_1\in\P^{d_1}:\,D_l\subset Z(t_1)\}$. Therefore $\mu_1(\P^{d_1}\setminus U_1)=0$. 

\par We assume now that $\nu_k(U_k)=1$ for any set $U_k$ as above. Let 
$$U_{k+1}=\{(t_1,\ldots,t_{k+1})\in\P^{d_1}\times\ldots\times\P^{d_{k+1}}:\,\dim Z(t_1)\cap\ldots\cap Z(t_{k+1})\cap A_j\leq n-k-1-j\}\,,$$
where $0\leq j\leq m-k-1$, $A_0=\emptyset$, or $A_j=\Bs V_{i_1}\cap\ldots\cap\Bs V_{i_j}$ with $k+2\leq i_1<\ldots<i_j\leq m$. Consider the set $U=U'\cap U''$, where  
$$U'=\{(t_1,\ldots,t_k)\in\P^{d_1}\times\ldots\times\P^{d_k}:\,\dim Z(t_1)\cap\ldots\cap Z(t_k)\cap A_j\leq n-k-j\},$$
$$U''=\{(t_1,\ldots,t_k)\in\P^{d_1}\times\ldots\times\P^{d_k}:\,\dim Z(t_1)\cap\ldots\cap Z(t_k)\cap\Bs V_{k+1}\cap A_j\leq n-k-j-1\}.$$
By the induction hypothesis we have $\nu_k(U')=\nu_k(U'')=1$, so $\nu_k(U)=1$. To prove that $\nu_{k+1}(U_{k+1})=1$ it suffices to show that 
$$\nu_{k+1}(W)=0\,,\,\text{ where } W:=(U\times\P^{d_{k+1}})\setminus U_{k+1}.$$ 
To this end we fix $t:=(t_1,\ldots,t_k)\in U$, we let  
$$Z(t):=Z(t_1)\cap\ldots\cap Z(t_k)\,,\;W(t):=\{t_{k+1}\in\P^{d_{k+1}}:\,\dim Z(t)\cap A_j\cap Z(t_{k+1})\geq n-k-j\},$$
and prove that $\mu_{k+1}(W(t))=0$.

\par Since $t\in U\subset U'$ we can write $Z(t)\cap A_j=\bigcup_{l=1}^ND_l\cup B$, where $D_l$ are the irreducible components of $Z(t)\cap A_j$ of dimension $n-k-j$ and $\dim B\leq n-k-j-1$. If $t_{k+1}\in W(t)$ then $Z(t)\cap A_j\cap Z(t_{k+1})$ is an analytic subset of $Z(t)\cap A_j$ of dimension $n-k-j$, so $D_l\subset Z(t)\cap A_j\cap Z(t_{k+1})$ for some $l$. Thus 
$$W(t)=\bigcup_{l=1}^NF_l(t)\,,\,\text{ where } F_l(t):=\{t_{k+1}\in\P^{d_{k+1}}:\,D_l\subset Z(t_{k+1})\}.$$
If $D_l\subset\Bs V_{k+1}$ then $\dim Z(t)\cap A_j\cap\Bs V_{k+1}=n-k-j$, which contradicts the fact that $t\in U''$. Hence the sections in $V_{k+1}$ cannot all vanish on $D_l$, so we may assume that $S_{k+1,d_{k+1}}\not\equiv0$ 
on $D_l$. We have $F_l(t)\subset\{t_{k+1,0}=0\}\cup H_l(t)$ where 
$$H_l(t):=\{[1:t_{k+1,1}:\ldots:t_{k+1,d_{k+1}}]\in\P^{d_{k+1}}:\,D_l\subset Z([1:t_{k+1,1}:\ldots:t_{k+1,d_{k+1}}])\}\,.$$
For each $(t_{{k+1},1}:\ldots:t_{{k+1},d_{k+1}-1})\in\C^{d_{k+1}-1}$ there exists at most one $\zeta\in\C$ with $[1:t_{k+1,1}:\ldots:t_{k+1,d_{k+1}-1}:\zeta]\in H_l(t)$. Indeed, if $\zeta\neq\zeta'$ have this property then 
$$S_{k+1,0}+t_{k+1,1}S_{k+1,1}+\ldots+t_{k+1,d_{k+1}-1}S_{k+1,d_{k+1}-1}+aS_{k+1,d_{k+1}}\equiv0\;\text{ on }D_l\,,$$
for $a=\zeta,\zeta'$, hence $S_{k+1,d_{k+1}}\equiv0$ on $D_l$, a contradiction. It follows that $\mu_{k+1}(H_l(t))=0$, so $\mu_{k+1}(F_l(t))=0$. Hence $\mu_{k+1}(W(t))=0$ and the proof is complete. 
\end{proof}

\par We return now to the setting of Theorem \ref{T:main1}. If $\{S^{k,p}_j\}_{j=0}^{d_{k,p}}$ is an orthonormal basis of $ H^0_{(2)}(X,L_k^p),$ we define the analytic hypersurface $Z(t_k)\subset X$, for $t_k=[t_{k,0}:\ldots:t_{k,d_{k,p}}]\in\P^{d_{k,p}}$, as in Proposition \ref{P:Bertini} $(ii)$. Let $\mu_{k,p}$ be the Fubini-Study volume on $\P^{d_{k,p}}$, $1\leq k\leq m$, $p\geq1$, and let $\mu_p=\mu_{1,p}\times\ldots\times\mu_{m,p}$ be the product measure on 
$\P^{d_{1,p}}\times\ldots\times\P^{d_{m,p}}$. Applying Proposition \ref{P:Bertini} we obtain:

\begin{proposition}\label{P:Bertini2}
In the above setting, if $p$ is sufficiently large then for $\mu_p$-a.e. $(t_1,\ldots,t_m)\in\P^{d_{1,p}}\times\ldots\times\P^{d_{m,p}}$ the analytic subsets $Z(t_1),\ldots,Z(t_m)\subset X$ are in general position, and $Z(t_{i_1})\cap\ldots\cap Z(t_{i_k})$ has pure dimension $n-k$ for each $1\leq k\leq m$, $1\leq i_1<\ldots<i_k\leq m$. 
\end{proposition}

\begin{proof}  Let $V_{k,p}:= H^0_{(2)}(X,L_k^p)$, so $\Bs V_{k,p}=\Sigma_{k,p}$. By Proposition \ref{P:FSwedge}, $\Sigma_{1,p},\ldots,\Sigma_{m,p}$ are in general position for all $p$ sufficiently large. We fix such $p$ and denote by $[Z(t_k)]$ the current of integration along the analytic hypersurface $Z(t_k)$; it has the same cohomology class as $pc_1(L_k,h_k)$. Proposition \ref{P:Bertini} shows that the analytic subsets $Z(t_1),\ldots,Z(t_m)$ are in general position for $\mu_p$-a.e. $(t_1,\ldots,t_m)\in\P^{d_{1,p}}\times\ldots\times\P^{d_{m,p}}$. Hence if $1\leq k\leq m$, $1\leq i_1<\ldots<i_k\leq m$, the current $[Z(t_{i_1})]\wedge\ldots\wedge[Z(t_{i_k})]$ is well defined by \cite[Corollary 2.11]{D93} and it is supported in $Z(t_{i_1})\cap\ldots\cap Z(t_{i_k})$. Since $c_1(L_k,h_k)\geq\varepsilon\omega,$  it follows that
$$\int_X[Z(t_{i_1})]\wedge\ldots\wedge[Z(t_{i_k})]\wedge\omega^{n-k}=p^k\int_Xc_1(L_{i_1},h_{i_1})\wedge\ldots\wedge c_1(L_{i_k},h_{i_k})\wedge\omega^{n-k}\geq p^k\varepsilon^k\int_X\omega^n.$$
So $Z(t_{i_1})\cap\ldots\cap Z(t_{i_k})\neq\emptyset$, hence it has pure dimension $n-k$. 
\end{proof}


\section{Convergence speed towards intersection of Fubini-Study currents}
\label{S:MT}

In this section we rely on techniques introduced by Dinh-Sibony \cite{DS06b}, based on the notion of meromorphic transform, in order to estimate the speed of equidistribution of the common zeros of $m$-tuples of sections of the considered big line bundles towards the intersection of the Fubini-Study currents. We then prove Theorem \ref{T:main1}. 

\subsection{Dinh-Sibony equidistribution theorem}\label{SS:Dinh-Sibony}
A meromorphic transform $F:\  X \dashrightarrow Y$ between two compact K\"ahler manifolds $(X,\omega)$ of dimension $n$ and $(Y,\omega_Y)$ of dimension $m$ 
is the data of an analytic  subset $\Gamma\subset X\times Y$ (called  {\it the graph of $F$}) 
of pure dimension $m+k$ such that the projections 
$\pi_1:\ X\times Y\to X$ and $\pi_2:\ X\times Y\to Y$ restricted  to each irreducible component 
of $\Gamma$ are surjective. We set formally $F=\pi_2\circ(\pi_1|_\Gamma)^{-1}$. 
For $y\in Y$ generic (that is, outside  a proper  analytic subset), the dimension of the fiber 
$F^{-1}(y):= \pi_1(\pi_2^{-1}|_\Gamma(y))$ is equal to $k$. 
This is called {\it the codimension} of $F$. We consider two of the {\it intermediate degrees} 
for $F$ (see \cite[Section\,3.1]{DS06b}): 
$$
d(F):=\int_X F^*(\omega_Y^m)\wedge \omega^k\quad \text{and}\quad \delta(F):= 
\int_X F^*(\omega_Y^{m-1})\wedge \omega^{k+1}.
$$

\par By \cite[Proposition 2.2]{DS06b}, there exists $r:=r(Y,\omega_Y)$ such that 
for every positive closed  current $T$ of bidegree $(1,1)$ on $Y$ with $\|T\|=1$ 
there is a  smooth $(1,1)$-form $\alpha$ which depends uniquely on the class $\{T\}$ 
and a quasiplurisubharmonic (qpsh) function $\varphi$ such that $-r\omega_Y\leq  \alpha\leq r\omega_Y$ 
and  $dd^c\varphi-T=\alpha$. If $Y$ is  the projective space $\P^\ell$ equipped with the Fubiny-Study form 
$\omega_{_{\FS}},$ then we have $r(\P^\ell,\omega_{_{\FS}})=1$. Consider the class
$$
Q(Y,\omega_Y):=\left\lbrace \varphi\,\text{ qpsh on } Y,\; dd^c \varphi\geq  -r(Y,\omega_Y)\omega_Y  \right\rbrace.
$$
A positive measure $\mu$ on $Y$ is called a BP measure if all qpsh functions on $Y$ are integrable with respect to $\mu$. 
When   $\dim Y=1,$ it is  well-known that $\mu$ is BP if and only if it admits locally  a bounded potential. 
The  terminology BP comes  from  this fact (see \cite{DS06b}).   

If $\mu$ is a BP measure on $Y$ and $t\in \R$, we let 
\begin{eqnarray*}
R(Y,\omega_Y,\mu)&:=&\sup\left\lbrace \max_Y\varphi:\,\varphi\in Q(Y,\omega_Y),\;\int_Y\varphi\,d\mu=0   \right\rbrace,\\
 \Delta( Y,\omega_Y,\mu,t)&:=& \sup\left\lbrace \mu(\varphi<-t):\, \varphi\in Q(Y,\omega_Y),\;\int_Y\varphi\,d\mu=0  \right\rbrace.
\end{eqnarray*}
These constants  are related to the Alexander-Dinh-Sibony capacity \cite{Al81,DS06b,GZ05}.
\par Let $\Phi_p$ be a sequence of meromorphic transforms from a compact K\"ahler manifold 
$(X, \omega)$ into compact K\"ahler manifolds $({\mathbb X}_p, \omega_p)$ of the same codimension $k,$
where ${\mathbb X}_p$  is defined in  \eqref{e:X_p}.  Let $\nu_p$ be a BP probability measure on 
${\mathbb X}_p$ and $\nu_\infty=\prod_{p\geq1}\nu_p$ be the product measure on 
$\Omega:=\prod_{p\geq1} {\mathbb X}_p$. For every $p>0$ and $\varepsilon>0$ let
$$E_p(\varepsilon):=\bigcup_{\|\phi\|_{\Cc^2}\leq 1}\left\lbrace x_p\in 
{\mathbb X}_p:\,\left|\left\langle\Phi_p^*(\delta_{x_p}) - \Phi_p^*(\nu_p),\phi\right\rangle\right| 
\geq d(\Phi_p)\varepsilon \right\rbrace,$$
where $\delta_{x_p}$ is the Dirac mass at $x_p$. Note that $\Phi_p^*(\delta_{x_p})$ 
and $\Phi_p^*(\nu_p)$ are positive closed currents of bidimension $(k,k)$ on $X$, 
and the former is well defined for the generic point $x_p\in{\mathbb X}_p$ (see \cite[Section 3.1]{DS06b}).  
Now we are in position to state the part which deals with the quantified speed of convergence 
in the Dinh-Sibony equidistribution theorem \cite[Theorem 4.1]{DS06b}.

\begin{theorem}[{\cite[Lemma 4.2 (d)]{DS06b}}] \label{T:Dinh-Sibony}
In the above setting the following  estimate holds: 
 $$\nu_p(E_p(\varepsilon))\leq \Delta\Big({\mathbb X}_p,\omega_p,\nu_p,\eta_{\varepsilon,p}\Big),$$
 where $\eta_{\varepsilon,p}:=\varepsilon\delta(\Phi_p)^{-1}d(\Phi_p)-
 3R({\mathbb X}_p,\omega_p,\nu_p)$.
 \end{theorem}

\subsection{Equidistribution of pullbacks of Dirac masses by Kodaira maps}

\par Let $(X,\omega)$ be a compact K\"ahler manifold of dimension $n$ and $(L_k,h_k)$, 
$1\leq k\leq m\leq n$, be singular Hermitian holomorphic line bundles on $X$ such that $h_k$ 
is continuous outside a proper analytic subset $\Sigma(h_k)\subset X$, 
$c_1(L_k,h_k)\geq\varepsilon\omega$ on $X$ for some $\varepsilon>0$, and $\Sigma(h_1),\ldots,\Sigma(h_m)$ 
are in general position. Recall  from Section \ref{introduction} that
 $$ {\mathbb X}_p:= \P H^0_{(2)}(X,L^p_1)\times\ldots\times\P H^0_{(2)}(X,L^p_m)\;,\;\;(\Omega,\sigma_\infty):=
 \prod_{p=1}^\infty({\mathbb X}_p,\sigma_p),$$
where the probability measure $\sigma_p$ is the product of the 
Fubini-Study volume on each factor. From now on let $p\in\N$ be large enough. 
Fix an orthonormal basis   $\{S^{k,p}_j\}_{j=0}^{d_{k,p}}$ as in \eqref{e:basis} 
and let $\Phi_{k,p}:\ X\dashrightarrow \P^{d_{k,p}}$ be the Kodaira map defined 
by this basis (see \eqref{e:Kodaira}). 
By \eqref{e:FSG} we have that $\Phi_{k,p}^*\omega_{_{\FS}}=\gamma_{k,p}$, 
where $\gamma_{k,p}$ is the Fubini-Study current of the space $H^0_{(2)}(X,L_k^p)$ 
as defined in \eqref{e:FSdef}.

\par We consider now the Kodaira maps as meromorphic transforms from $X$ to 
$ {\mathbb P}H^0_{(2)}(X,L^p_k)$ which we denote still by 
$\Phi_{k,p}:\ X \dashrightarrow{\mathbb P}H^0_{(2)}(X,L^p_k)$. 
Precisely, this is the meromorphic transform with graph 
$$\Gamma_{k,p}=\big\{(x,s)\in X\times{\mathbb P}H^0_{(2)}(X,L^p_k):\,s(x)=0\big\},\;1\leq k\leq m\,.$$
Indeed, since $\dim H^0_{(2)}(X,L^p_k)\geq 2$ (see e.g. Proposition \ref{P:Bsdim} below), 
there exists, for every $x\in X,$ a section $s\in H^0_{(2)}(X,L^p_k)$ with $s(x)=0$, 
so the projection $\Gamma_{k,p}\longrightarrow X$ is surjective. Moreover, 
since $L_k^p$ is non-trivial, every global holomorphic section of $L_k^p$ 
must vanish at some $x\in X$, hence the projection 
$\Gamma_{k,p}\longrightarrow{\mathbb P}H^0_{(2)}(X,L^p_k)$ is surjective. Note that 
$$\Phi_{k,p}(x)=\big\{s\in{\mathbb P}H^0_{(2)}(X,L^p_k):
\,s(x)=0\big\},\;\;\Phi_{k,p}^{-1}(s)=\big\{x\in X:\,s(x)=0\big\}\,.$$
Let $\Phi_p$ be the product transform of $\Phi_{1,p},\ldots,\Phi_{m,p}$ 
(see \cite[Section 3.3]{DS06b}). It is the meromorphic transform with graph
\begin{equation}\label{e:Gamma_p}
\Gamma_p=\big\{(x,s_{p1},\ldots,s_{pm})\in X\times {\mathbb X}_p:\  s_{p1}(x)=\ldots=s_{pm}(x)=0\big\}\,.
\end{equation}
By above, the projection $\Pi_1:\Gamma_p\longrightarrow X$ is surjective. 
The second projection $\Pi_2:\Gamma_p\longrightarrow{\mathbb X}_p$ is proper, 
hence by Remmert's theorem $\Pi_2(\Gamma_p)$ is an analytic subvariety of ${\mathbb X}_p$. 
Proposition \ref{P:Bertini2} implies that $\Pi_2(\Gamma_p)$ has full measure in ${\mathbb X}_p$, 
so $\Pi_2$ is surjective and $\Phi_p$ is a meromorphic transform of codimension $n-m$, with fibers 
$$\Phi^{-1}_p(\bfs_p)=\{x\in X:\ s_{p1}(x)=\ldots=s_{pm}(x)=0\}\,,
\,\text{ where }\,\bfs_p=(s_{p1},\ldots,s_{pm})\in{\mathbb X}_p\,.$$
Considering the product transform of any $\Phi_{i_1,p},\ldots,\Phi_{i_k,p}$, $1\leq i_1<\ldots<i_k\leq m$, 
and arguing as above it follows that, for $\bfs_p=(s_{p1},\ldots,s_{pm})\in{\mathbb X}_p$ generic, 
the analytic sets $\{s_{p1}=0\},\ldots,\{s_{pm}=0\}$ are in general position. 
Hence by \cite[Corollary 2.11]{D93} the following current of bidegree $(m,m)$ is well defined on $X$:
$$\Phi_p^\ast(\delta_{\bfs_p})=[\bfs_p=0]=[s_{p1}=0]\wedge\ldots\wedge[s_{pm}=0]=
\Phi_{1,p}^*(\delta_{s_{p1}})\wedge\ldots\wedge\Phi_{m,p}^*(\delta_{s_{pm}})\,.$$

\par The main result of this section is the following theorem.

\begin{theorem}\label{T:speed} Under the hypotheses of Theorem \ref{T:main1} 
there exist a constant $\xi>0$ depending only on $m$ and a constant $c=c(X,L_1,h_1,\ldots,L_m,h_m)>0$ 
with the following property: For any sequence of positive numbers $\{\lambda_p\}_{p\geq1}$ with 
$$\liminf_{p\to\infty}\frac{\lambda_p}{\log p}>(1+\xi n)c,$$ 
there are subsets $E_{p}\subset {\mathbb X}_p$ such that 
\\[2pt]
(a) $\sigma_p(E_{p})\leq  c p^{\xi n} \exp(-\lambda_p/c)$ for all $p$ large enough;
\\[3pt]
(b) if $\bfs_p\in  {\mathbb X}_p\setminus E_{p}$ we have that the estimate 
$$\left|\frac{1}{p^m}\big\langle[\bfs_p=0]-\gamma_{1,p}\wedge\ldots\wedge\gamma_{m,p}\,,\phi\big\rangle\right|
\leq c\,\frac {\lambda_p}{p}\,\|\phi\|_{\Cc^2}$$
holds for every $(n-m,n-m)$ form $\phi$ of class $\Cc^2$. 

\par In particular, for $\sigma_\infty$-a.e. $\bfs\in \Omega$ the estimate from (b) holds for all $p$ sufficiently large.
\end{theorem}
Prior to the  proof we need  to  establish some preparatory results. Let 
$$d_{0,p}=d_{1,p}+\ldots+d_{m,p}$$
be the  dimension of ${\mathbb X}_p$ and $\pi_k$ be the canonical projection of ${\mathbb X}_p$ 
onto its $k$-th factor. Let 
$$\omega_p:= c_p \big(\pi_1^*\omega_{_{\FS}}+\ldots+\pi_m^*\omega_{_{\FS}}\big)\,,\,
\text{ so }\,\sigma_p=\omega_p^{d_{0,p}}\,.$$
Here $\omega_{_{\FS}}$ denotes, as  usual, the Fubini-Study form on each factor 
$\P H^0_{(2)}(X,L^p_k)$, and the constant  $c_p$ is chosen so that $\sigma_p$ 
is a probability measure on ${\mathbb X}_p$, thus 
\begin{equation}\label{e:c_p}
(c_p)^{-d_{0,p}}=  \frac{d_{0,p}!}{d_{1,p}!\ldots d_{m,p}! }\,\cdot
\end{equation}

\begin{lemma}\label{L:c_0}
There is a constant $c_0>0$ such that $c_p\geq  c_0$ for all $p\geq 1.$
\end{lemma}

\begin{proof} Fix  $p\geq1$ large enough. For each $ 1\leq k\leq m,$ let $l_k:=d_{k,p}$.
Using Stirling's formula  $\ell!\approx(\ell/e)^\ell \sqrt{2\pi\ell}$ it suffices to show 
that there is a constant $c>0$ such that for all $l_1,\ldots,l_m\geq1$,
$$
\log{(l_1+\ldots +l_m)}-  
\Big(\frac{ l_1\log l_1}{l_1+\ldots+l_m}+ \ldots+ \frac{l_m\log l_m}{l_1+\ldots+l_m}  \Big)\leq  c.
$$
Since the function $t\mapsto  t\log t$, $t>0$, is convex, we infer that
$$ {\frac1m}\Big (l_1 \log l_1+ \ldots+ l_m \log l_m \Big)\geq \frac{l_1+\ldots+l_m}{m}\,
\log\frac{l_1+\ldots+l_m}{m}\;\cdot$$
This implies the required estimate with $c:=\log m$.
\end{proof}

\noindent
Following subsection \ref{SS:Dinh-Sibony} we consider two intermediate degrees for the Kodaira maps $\Phi_p$:
$$
d_p=d(\Phi_p):=\int_X \Phi_p^*(\omega_p^{d_{0,p}})\wedge \omega^{n-m}\quad 
\text{and}\quad \delta_p=\delta(\Phi_p):= \int_X \Phi_p^*(\omega_p^{d_{0,p}-1})\wedge \omega^{n-m+1}.
$$
The next result gives the asymptotic behavior of $d_p$ and $\delta_p$ as $p\to\infty$.

\begin{lemma}\label{L:d_p_delta_p}
We have $d_p=p^m\|  c_1(L_1,h_1)\wedge\ldots\wedge c_1(L_m,h_m)\|$ and  
$$\delta_p=\frac{p^{m-1}}{c_p}\;\sum_{k=1}^m\frac{d_{k,p}}{d_{0,p}}\;
\Big\|\bigwedge_{l=1,l\neq k}^mc_1(L_l,h_l)\Big\|\leq Cp^{m-1}\,,$$
where $C>0$ is a constant depending on $(L_k,h_k)\,$, $1\leq k\leq m$.
\end{lemma}

\begin{proof}
We use a cohomological argument. For the first identity we replace  $\omega_p^{d_{0,p}}$ by 
a Dirac mass $\delta_\bfs,$ where $\bfs:=(s_1,\ldots,s_m)\in{\mathbb X}_p$ is such that 
$\{s_1=0\},\ldots,\{s_m=0\}$ are in general position, so the current $\Phi_p^*(\delta_\bfs)=
[s_1=0]\wedge\ldots\wedge[s_m=0]$ is well defined (see Proposition \ref{P:Bertini2}). 
By the Poincar\'e -Lelong formula \cite[Theorem\,2.3.3]{MM07}, 
$$[s_k=0]=pc_1(L_k,h_k)+dd^c\log|s_k|_{h_{k,p}}\,,\;1\leq k\leq m.$$
Since the current $c_1(L_1,h_1)\wedge\ldots\wedge c_1(L_m,h_m)$ is well defined 
(see Proposition \ref{P:FSwedge}) it follows that 
$$\int_X\Phi_p^*(\delta_\bfs)\wedge\omega^{n-m}=
p^m\int_X\theta_1\wedge\ldots\wedge\theta_m\wedge\omega^{n-m}=
p^m\int_Xc_1(L_1,h_1)\wedge\ldots\wedge c_1(L_m,h_m)\wedge\omega^{n-m},$$
where $\theta_k$ is a smooth closed $(1,1)$ form in the cohomology class of $c_1(L_k,h_k)$. Thus 
$$d_p=\int_X \Phi_p^*(\omega_p^{d_{0,p}})\wedge\omega^{n-m}=
\int_X \Phi_p^*(\delta_\bfs)\wedge\omega^{n-m}=p^m\|  c_1(L_1,h_1)\wedge\ldots\wedge c_1(L_m,h_m)\|.$$

\par For the second identity, a straightforward computation shows that
$$\omega_p^{d_{0,p}-1}=\sum_{k=1}^m\frac{c_p^{d_{0,p}-1}(d_{0,p}-1)!}{d_{1,p}!\ldots(d_{k,p}-1)!\ldots d_{m,p}!}\;
\pi_1^*\omega_{_{\FS}}^{d_{1,p}}\wedge\ldots
\wedge\pi_k^*\omega_{_{\FS}}^{d_{k,p}-1}\wedge\ldots\wedge\pi_m^*\omega_{_{\FS}}^{d_{m,p}}.$$
Using (\ref{e:c_p}) and  replacing  $\omega_{_{\FS}}^{d_{k,p}}$ (resp.\ $\omega_{_{\FS}}^{d_{k,p}-1}$) 
by a generic point  (resp.\ a generic complex line) in $\P H^0_{(2)}(X,L^p_k)$, 
we may replace $\omega_p^{d_{0,p}-1}$ by a current of the form 
$$T:=\sum_{k=1}^m\frac{d_{k,p}}{c_p d_{0,p}}\;[\{s_1\}\times\ldots\times{\mathcal D_k}\times\ldots\times\{s_m\}].$$  
Here, ${\mathcal D_k}$ is a generic  complex line in $ \P H^0_{(2)}(X,L^p_k)$ and $(s_1,\ldots,s_m)$ 
is a generic point in ${\mathbb X}_p$. The genericity of $\mathcal D_k$ implies that $\Phi_{k,p}^*({\mathcal D_k})=X$, 
so
$$\Phi_p^*([\{s_1\}\times\ldots\times{\mathcal D_k}\times\ldots\times\{s_m\}])=
\bigwedge_{l=1,l\neq k}^m[s_l=0]\,.$$ 
The Poincar\'e-Lelong formula yields 
$$\|\Phi_p^*([\{s_1\}\times\ldots\times{\mathcal D_k}\times\ldots\times\{s_m\}])\|=
p^{m-1}\Big\|\bigwedge_{l=1,l\neq k}^mc_1(L_l,h_l)\Big\|.$$
Since $\delta_p=\|\Phi_p^*(T)\|$, the second identity follows. Using Lemma \ref{L:c_0}, this yields the upper bound on $\delta_p$. 
\end{proof}

\begin{lemma}\label{L:KTFS}
For all $p$ sufficiently large we have $\,\Phi^*_p(\sigma_p)=\gamma_{1,p}\wedge\ldots\wedge\gamma_{m,p}\,$.
\end{lemma}

\begin{proof} Let us write ${\mathbb X}_p=X_{1,p}\times\ldots\times X_{m,p}$ 
and $\sigma_p=\sigma_{1,p}\times\ldots\times\sigma_{m,p}$, where $X_{k,p}=
\P H^0_{(2)}(X,L^p_k)$ and $\sigma_{k,p}$ is the Fubini-Study volume on $X_{k,p}$. 
Recall that the meromorphic transform $\Phi_p$ has graph $\Gamma_p$ defined in 
\eqref{e:Gamma_p}, and $\Pi_1:\Gamma_p\longrightarrow X$, $\Pi_2:\Gamma_p\longrightarrow{\mathbb X}_p$, 
denote the canonical projections. By the definition of $\Phi^*_p(\sigma_p)$ (see \cite[Sect. 3.1]{DS06b}) we have 
$$\langle\Phi^*_p(\sigma_p),\phi\rangle=
\int_{\Gamma_p}\Pi_1^*(\phi)\wedge\Pi_2^*(\sigma_p)=
\int_{{\mathbb X}_p}\Pi_{2*}\Pi_1^*(\phi)\wedge\sigma_p=
\int_{{\mathbb X}_p}\langle[\bfs_p=0],\phi\rangle\,d\sigma_p(\bfs_p),$$
where $\phi$ is a smooth $(n-m,n-m)$ form on $X$. Thanks to Propositions \ref{P:FSwedge} 
and \ref{P:Bertini}, we can apply \cite[Proposition 4.2]{CM11} as in the proof of \cite[Theorem 1.2]{CM11} to show that 
\begin{eqnarray*}
\langle\Phi^*_p(\sigma_p),\phi\rangle&=&\int_{X_{m,p}}\ldots\int_{X_{1,p}}\langle[s_{p1}=0]
\wedge\ldots\wedge[s_{pm}=0],\phi\rangle\,d\sigma_{1,p}(s_{p1})\ldots d\sigma_{m,p}(s_{pm})\\
&=&\int_{X_{m,p}}\ldots\int_{X_{2,p}}\langle\gamma_{1,p}\wedge[s_{p2}=0]\wedge\ldots
\wedge[s_{pm}=0],\phi\rangle\,d\sigma_{2,p}(s_{p2})\ldots d\sigma_{m,p}(s_{pm})\\
&=&\ldots\;=\langle\gamma_{1,p}\wedge\ldots\wedge\gamma_{m,p},\phi\rangle\,.
\end{eqnarray*}
This concludes the proof of the lemma. 
\end{proof}

\begin{lemma}\label{L:Dinh-Sibony}
There exist absolute constants $C_1,\alpha>0$, and constants $C_2,\alpha',\xi>0$ 
depending only on $m\geq1$, such that for all $\ell,\ell_1,\ldots,\ell_m\geq 1$ and $t\geq0$,
\begin{eqnarray*}
R(\P^\ell,\omega_{_{\FS}},\omega_{_{\FS}}^\ell)&\leq & \frac{1}{2}\,(1+\log\ell)\,,\\
\Delta(\P^\ell,\omega_{_{\FS}},\omega_{_{\FS}}^\ell,t)&\leq & C_1\ell\,e^ {-\alpha t}\,,\\
r(\P^{\ell_1}\times\ldots\times\P^{\ell_m},\omega_{_\MP})&\leq&r(\ell_1,\ldots,\ell_m):=\max_{1\leq k\leq m}\frac{d}{\ell_k}\;,\\
R(\P^{\ell_1}\times\ldots\times\P^{\ell_m},\omega_{_\MP},\omega_{_\MP}^d)&\leq & C_2r(\ell_1,\ldots,\ell_m)(1+\log d)\,,\\
\Delta(\P^{\ell_1}\times\ldots\times\P^{\ell_m},\omega_{_\MP},\omega_{_\MP}^d,t)
&\leq & C_2d^{\,\xi}e^{-\alpha't/r(\ell_1,\ldots,\ell_m)}\,,
\end{eqnarray*}
where 
$$d=\ell_1+\ldots+\ell_m\,,\;\omega_{_\MP}:=c\big(\pi_1^\ast(\omega_{_{\FS}})+
\ldots+\pi_m^\ast(\omega_{_{\FS}})\big)\,,\;c^{-d}=\frac{d!}{\ell_1!\ldots\ell_m!}\,,$$ 
so $\omega_{_\MP}^d$ is a probability measure on $\P^{\ell_1}\times\ldots\times\P^{\ell_m}$.  
\end{lemma}

\begin{proof} The first two inequalities are proved in Proposition A.3 and Corollary A.5 
from \cite{DS06b}. If $T$ is a positive closed current of bidegree $(1,1)$ on 
$\P^{\ell_1}\times\ldots\times\P^{\ell_m}$ with $\|T\|=1$ then $T$ is in the cohomology class 
of $\alpha=a_1\pi_1^\ast(\omega_{_{\FS}})+\ldots+a_m\pi_m^\ast(\omega_{_{\FS}})$, 
for some $a_k\geq0$. Hence 
$$0\leq\alpha\leq\left(\max_{1\leq k\leq m}\frac{a_k}{c}\right)\omega_{_\MP}\,.$$
Now 
$$1=\|T\|=\int_{\P^{\ell_1}\times\ldots\times\P^{\ell_m}}\alpha\wedge\omega_{_\MP}^{d-1}
=\sum_{k=1}^m\frac{a_k\ell_k}{cd}\,,$$
so $a_k/c\leq d/\ell_k$. Thus $r(\P^{\ell_1}\times\ldots\times\P^{\ell_m},\omega_{_\MP})
\leq\max_{1\leq k\leq m}\frac{d}{\ell_k}$. 
The last two inequalities follow from these estimates by applying \cite[Proposition A.8, Proposition A.9]{DS06b}.
\end{proof}

\par We will also need the following lower estimate for the dimension $d_{k,p}$.

\begin{proposition}\label{P:Bsdim}
Let $(X,\omega)$ be a compact K\"ahler manifold of dimension $n$. Let $(L,h)\to X$ be a singular
Hermitian holomorphic line bundle such that $c_1(L,h)\geq\varepsilon\omega$ for some
$\varepsilon>0$ and $h$ is continuous outside a proper analytic subset of $X$.
Then there exists $C>0$ and $p_0\in\N$ such that
\[
\dim H^0_{(2)}(X,L^p)\geq Cp^n\,,\:\:\forall\,p\geq p_0.
\]
\end{proposition}

\begin{proof}
Let $\Sigma\subset X$ be a proper analytic set such that $h$ is continuous on $X\setminus\Sigma$.
We fix $x_0\in X\setminus\Sigma$ and $r>0$ such that $B(x_0,2r)\cap\Sigma=\emptyset$.
Let $0\leq\chi\leq1$ be a smooth cut-off function which equals $1$ on $\overline B(x_0,r)$
and is supported in $B(x_0,2r)$. 
We consider the function $\psi:X\to[-\infty,\infty)$, $\psi(x)=\eta\chi(x)\log|x-x_0|$, where $\eta>0$.

\par Consider the metric $h_0=h\exp(-\psi)$ on $L$. We choose $\eta$ sufficiently small such that
\[
c_1(L,h_0)\geq\frac{\varepsilon}{2}\,\omega \:\:\text{on $X$}.
\]
Let us denote by $\mI(h^p)$ the multiplier ideal sheaf associated with $h^p$. 
Note that $H^0_{(2)}(X,L^p)=H^0(X,L^p\otimes\mI(h^p))$. 
The Nadel vanishing theorem \cite{D93b, Nad89} shows that there exists $p_0\in\N$ such that
\begin{equation}\label{Nvth}
H^1(X,L^p\otimes\mI(h_0^p))=0\,,\:\:p\geq p_0.
\end{equation}
Note that $\mI(h_0^p)=\mI(h^p)\otimes\mI(p\psi)$. Consider the exact sequence
\begin{equation}\label{shs}
0\to L^p\otimes \mI(h^p)\otimes\mI(p\psi)\to L^p\otimes \mI(h^p)\to
L^p\otimes \mI(h^p)\otimes \mO_X/\mI(p\psi)\to 0\,.
\end{equation}
Thanks to \eqref{Nvth} applied to the long exact cohomology sequence associated with \eqref{shs}
we have
\begin{equation}\label{lhs}
H^0(X,L^p\otimes \mI(h^p))\to
H^0(X,L^p\otimes \mI(h^p)\otimes \mO_X/\mI(p\psi))\to 0\,,\:\:p\geq p_0.
\end{equation}
Now,  for $x\neq x_0$, $\mI(p\psi)_x=\mO_{X,x}$ hence $\mO_{X,x}/\mI(p\psi)_x=0$.
Moreover $\mI(h^p)_{x_0}=\mO_{X,x_0}$ since $h$ is continuous at $x_0$.
Hence
\begin{equation}\label{e}
\begin{split}
H^0(X,L^p\otimes \mI(h^p)\otimes \mO_X/\mI(p\psi))&=
L^p_{x_0}\otimes \mI(h^p)_{x_0}\otimes \mO_{X,x_0}/\mI(p\psi)_{x_0}\\
&=
L^p_{x_0}\otimes\mO_{X,x_0}/\mI(p\psi)_{x_0},
\end{split}
\end{equation}
so 
\begin{equation}\label{shs1}
H^0(X,L^p\otimes \mI(h^p))\to
L^p_{x_0}\otimes\mO_{X,x_0}/\mI(p\psi)_{x_0}\to 0\,,\:\:p\geq p_0.
\end{equation}
Denote by $\mathcal{M}_{X,x_0}$ the maximal ideal of  $\mathcal{O}_{X,x_0}$ (that is, 
germs of holomorphic functions vanishing at $x_0$).
We have $\mI(p\psi)_{x_0}\subset\mathcal{M}_{X,x_0}^{[p\eta]-n+1}$
and $\dim \mO_{X,x_0}/\mathcal{M}_{X,x_0}^{k+1}=\binom{k+n}{k}$,
which together with \eqref{shs1} implies the conclusion.
\end{proof}

\smallskip

\begin{proof}[Proof of Theorem \ref{T:speed}] We will apply Theorem \ref{T:Dinh-Sibony} 
to the meromorphic transforms $\Phi_p$ from $X$ to the multi-projective space $({\mathbb X}_p,\omega_p)$ 
defined above, and the BP measures $\nu_p:=\sigma_p$ on ${\mathbb X}_p$. For $t\in \R$ and $\varepsilon>0$ let
\begin{eqnarray}
&&R_p:=R({\mathbb X}_p,\omega_p,\sigma_p)\;,\;\;\Delta_p(t):=   
\Delta({\mathbb X}_p,\omega_p,\sigma_p,t)\,,\nonumber\\
&&E_p(\varepsilon):=\bigcup_{\|\phi\|_{\Cc^2}\leq 1}\big\lbrace \bfs\in 
{\mathbb X}_p:\  \big|\big\langle[\bfs=0]-\gamma_{1,p}\wedge\ldots\wedge\gamma_{m,p}\,,
\phi\big\rangle\big|\geq d_p\varepsilon \big\rbrace\label{e:E_p_epsilon}\,.
\end{eqnarray}
It follows from Siegel's lemma \cite[Lemma 2.2.6] {MM07} and Proposition \ref{P:Bsdim} 
that there exist $C_3>0$ depending only on $(X,L_k,h_k)_{1\leq k\leq m}$ and $p_0\in\mathbb N$ such that 
$$p^n/C_3\leq d_{k,p}\leq C_3p^n\,,\;p\geq p_0\,,\;1\leq k\leq m.$$ 
By the last two inequalities in Lemma \ref{L:Dinh-Sibony} we obtain for $p\geq p_0$ and $t\geq0$, 
\begin{equation}\label{e:c'}
\begin{split}
R_p&\leq mC_2C_3^2(1+ \log(mC_3p^n))\leq C_4\log p\,,\\
\Delta_p(t)&\leq C_2(mC_3p^n)^\xi\exp\left(\frac{-\alpha't}{mC_3^2}\right)\leq 
C_4p^{\xi n}e^{-t/C_4}\,,
\end{split}
\end{equation}
where $C_4$ is a constant depending only on $(X,L_k,h_k)_{1\leq k\leq m}$. Now set 
$$\varepsilon_p:=\lambda_p/p\;,\;\;\eta_p:=\varepsilon_pd_p/\delta_p-3R_p\,.$$
Lemma \ref{L:d_p_delta_p} implies that $d_p\approx p^m$, $\delta_p\lesssim p^{m-1}$, so 
$$\eta_p\geq C_5\lambda_p-3C_4\log p\,,\;p\geq p_0\,,$$
where $C_5$ is a constant depending only on $(X,L_k,h_k)_{1\leq k\leq m}$. Note that for all $p$ sufficiently large, 
$$\eta_p>\frac{C_5}{2}\,\lambda_p\,,\,\text{ provided that }\liminf_{p\to\infty}\frac{\lambda_p}{\log p}>6C_4/C_5\,.$$ 
If $E_p=E_p(\varepsilon_p)$ then it follows from Theorem \ref{T:Dinh-Sibony} and 
Lemma \ref{L:KTFS} that for all $p$ sufficiently large 
$$\sigma_p(E_p)\leq\Delta_p(\eta_p)\leq C_4 p^{\xi n}\exp\Big(\frac{-C_5\lambda_p}{2C_4}\Big),$$
where for the last estimate we used \eqref{e:c'}. Let 
$$c=\max\left(\frac{6C_4}{C_5(1+\xi n)}\,,\frac{2C_4}{C_5}\,,C_4,\|c_1(L_1,h_1)\wedge\ldots\wedge c_1(L_m,h_m)\|\right).$$
If $\liminf_{p\to\infty}(\lambda_p/\log p)>(1+\xi n)c$ then for all $p$ sufficiently large 
$$\sigma_p(E_p)\leq C_4p^{\xi n}\exp\Big(\frac{-C_5\lambda_p}{2C_4}\Big)
\leq c\,p^{\xi n}\exp\Big(\frac{-\lambda_p}{c}\Big).$$
On the other hand we have by the definition of $E_p$ that if 
$\bfs_p\in{\mathbb X}_p\setminus E_{p}$ and $\phi$ is an $(n-m,n-m)$ form of class $\Cc^2$ then 
$$\left|\frac{1}{p^m}\left \langle[\bfs_p=0]- \gamma_{1,p}\wedge\ldots\wedge\gamma_{m,p}\,,
\phi\right\rangle\right|\leq\frac{d_p}{p^m}\,\frac {\lambda_p}{p}\,
\|\phi\|_{\Cc^2}\leq c\,\frac {\lambda_p}{p}\,\|\phi\|_{\Cc^2}\,.$$
In the last inequality we used the fact that $d_p\leq c\,p^m$ by Lemma \ref{L:d_p_delta_p}.

\par For the last conclusion of Theorem \ref{T:speed} we proceed as in \cite[p.\ 9]{DMM14}. 
The assumption on $\lambda_p/\log p$ and $(a)$ imply that 
$$
\sum_{p=1}^\infty \sigma_p(E_p)\leq c'\sum_{p=1}^\infty \frac{1}{p^\eta}<\infty
$$ 
for some $c'>0$ and $\eta>1$. Hence the set
$$
E:=\big\lbrace  \bfs =(\bfs_1,\bfs_2,\ldots)\in\Omega:\ \bfs_p\in E_p
 \ \text{for infinitely many}\ p \big\rbrace
$$
satisfies $\sigma_\infty(E)=0$. Indeed, for every  $N\geq 1$, $E$ is contained in the set  
$$
\big\lbrace  \bfs =(\bfs_1,\bfs_2,\ldots)\in\Omega:\ \bfs_p\in E_p \ \text{for at least one}
\ p\geq N \big\rbrace,
$$
whose $\sigma_\infty$-measure is at most 
$$
\sum_{p=N}^\infty \sigma_p(E_p)\leq c'\sum_{p=N}^\infty \frac{1}{p^\eta}\to 0\;\;\text{as } N\to\infty.
$$
The proof of the theorem  is thereby completed.
 \end{proof}
 
\medskip

\begin{proof}[Proof of Theorem \ref{T:main1}] Theorem \ref{T:main1} follows directly from 
Theorem \ref{T:speed} and Proposition \ref{P:FSwedge} $(iii)$.
\end{proof}

 
\section{Equidistribution with convergence speed for H\"older singular metrics}
\label{S:PT:main2&3}
 
\par In this section we prove Theorem \ref{T:main2}. We close with more examples of 
H\"older metrics with singularities. Theorem \ref{T:main2} follows at once from Theorem \ref{T:speed} and the next result.

\begin{theorem}\label{T:FSspeed1}  
In the setting of Theorem \ref{T:main2}, there exists a constant $c>0$ depending only on 
$(X,L_1,h_1,\ldots,L_m,h_m)$ such that for all $p$ sufficiently large the estimate 
$$\Big|\Big\langle\frac{1}{p^m}\,\bigwedge_{k=1}^m\gamma_{k,p}-\bigwedge_{k=1}^mc_1(L_k,h_k),\phi\Big\rangle\Big|
\leq \frac{c\log p}{p}\, \| \phi\|_{\Cc^2}$$
holds for every $(n-m,n-m)$ form $\phi $ of class $\Cc^2$.  
\end{theorem}

\begin{proof} We may assume that $\phi$ is real. 
There exists a constant $c'>0$ such that for every real $(n-m,n-m)$ form $\phi $ of class $\Cc^2$, 
\begin{equation}\label{e:normpsi}
-c' \|\phi\|_{\Cc^2}\,\omega^{n-m+1}\leq  dd^c\phi\leq c' \|\phi\|_{\Cc^2}\,\omega^{n-m+1}.
\end{equation}
Using Proposition \ref{P:FSwedge} and \eqref{e:FSB} we can write 
$$\Big\langle\frac{1}{p^m}\,\bigwedge_{k=1}^m\gamma_{k,p}-\bigwedge_{k=1}^mc_1(L_k,h_k),\phi\Big\rangle
=\sum_{k=1}^mI_k\,,$$
where
\begin{eqnarray*}
I_k&=&\Big\langle c_1(L_1,h_1)\wedge\ldots\wedge c_1(L_{k-1},h_{k-1})
\wedge\Big(\frac{\gamma_{k,p}}{p}-c_1(L_k,h_k)\Big)\wedge\frac{\gamma_{k+1,p}}{p}
\wedge\ldots\wedge\frac{\gamma_{m,p}}{p}\,,\phi\Big\rangle\\
&=&\Big\langle c_1(L_1,h_1)\wedge\ldots\wedge c_1(L_{k-1},h_{k-1})\wedge\frac{dd^c\log P_{k,p}}{2p}
\wedge\frac{\gamma_{k+1,p}}{p}\wedge\ldots\wedge\frac{\gamma_{m,p}}{p}\,,\phi\Big\rangle\\
&=&\int_X\frac{\log P_{k,p}}{2p}\;c_1(L_1,h_1)\wedge\ldots\wedge c_1(L_{k-1},h_{k-1})
\wedge\frac{\gamma_{k+1,p}}{p}\wedge\ldots\wedge\frac{\gamma_{m,p}}{p}\wedge dd^c\phi\,.
\end{eqnarray*}
By \eqref{e:normpsi} the total variation of the measure in the last integral is dominated by the 
positive measure $c' \|\phi\|_{\Cc^2}\,\mu$, where 
$$\mu:=c_1(L_1,h_1)\wedge\ldots\wedge c_1(L_{k-1},h_{k-1})
\wedge\frac{\gamma_{k+1,p}}{p}\wedge\ldots\wedge\frac{\gamma_{m,p}}{p}\wedge\omega^{n-m+1}\,,$$
hence
$$|I_k|\leq c' \|\phi\|_{\Cc^2}\int_X\frac{|\log P_{k,p}|}{2p}\;d\mu\,.$$
Theorem \ref{T:Bka} implies that there exist a constant $c''>0$ and $p_0\in\mathbb N$ 
such that for all $z\in X\setminus\Sigma(h_k)$ and all $p\geq p_0$ one has 
 $$-c''\leq\log P_{k,p}\leq c''\log p+ c''\big|\log\dist \big(z,\Sigma(h_k)\big)\big|\,,\;1\leq k\leq m\,.$$ 
 We obtain that 
$$|I_k|\leq\frac{c'c''\|\phi\|_{\Cc^2}}{2p}
\int_X\,\big(\log p +\big|\log\dist \big(z,\Sigma(h_k)\big)\big|\big)\;d\mu\,,\;\,\forall\,p\geq p_0\,,1\leq k\leq m\,.$$
Applying  Lemma  \ref{L:main_estimate} below to the  right hand side yields that
$$|I_k|\leq \frac{c\log p}{mp}\,\|\phi\|_{\Cc^2} \,,\;\,\forall\,p\geq p_0\,,1\leq k\leq m\,,$$ 
with some constant $c=c(X,L_1,h_1,\ldots,L_m,h_m)>0.$ The proof is thereby completed.
\end{proof}

\medskip

\par  The following crucial  estimate  was used in the proof of Theorem \ref{T:FSspeed1}:

\begin{lemma}\label{L:main_estimate}
In the setting of Theorem \ref{T:FSspeed1} there exists
a constant  $C>0$ such that for every $1\leq k\leq m,$
$$
\int_X \big|\log\dist \big(z,\Sigma(h_k)\big)\big|\,c_1(L_1,h_1)
\wedge\ldots\wedge c_1(L_{k-1},h_{k-1})\wedge\frac{\gamma_{k+1,p}}{p}
\wedge\ldots\wedge\frac{\gamma_{m,p}}{p}\wedge \omega^{n-m+1}
<C.
$$
\end{lemma}

\begin{proof}
Let $U\subset X$ be a contractible Stein open set as in Section \ref{S:FSB}. 
For $1\leq j\leq m$ and $p\geq 1$, let $u_{j,p}\,,\,u_j$ be the psh functions defined 
in \eqref{e:FSpot}, so $dd^cu_j=c_1(L_j,h_j)$ and $dd^cu_{j,p}=\frac{1}{p}\,\gamma_{j,p}$ on $U.$ 

By shrinking $U$ if necessary, we may construct a negative psh function $v_k$ on 
$U$ such that  $v_k\leq \log\dist\big(z,\Sigma(h_k)\big)<0$
and $v_k$ is  smooth outside $\Sigma(h_k)$.  Indeed, let  $f_1,\ldots f_N$ be  holomorphic  functions defined
on a neighborhood of $\overline U$
such that 
$$\Sigma(h_k)\cap U=\{z\in U:\  f_1(z)=\ldots=f_N(z)=0\}.$$
We see  easily that  the function $v_k(z):= \log \big(|f_1(z)|^2+\ldots+ |f_N(z)|^2\big)-c'$, $z\in U$, 
with a suitable constant $c'>0$, does the job. Indeed, we may assume that the function 
$h(z)=|f_1(z)|^2+\ldots+ |f_N(z)|^2$ is Lipschitz on $U$, so there exist constants $c_1,c_2>0$
such that
in local coordinates $z$ on $U$, we have $|h(z)|\leq c_1|z-w|$ for all $z\in U$, $w\in\Sigma(h_k)$, 
hence $|h(z)|\leq c_2\dist(z,\Sigma(h_k))$, so $\log|h(z)|\leq\log\dist(z,\Sigma(h_k))+\log c_2$\,. 
  
By \cite[Theorem 5.1]{CM11} we have that $\frac{1}{p}\,\log P_{j,p}\to 0$ 
in $L^1(X,\omega^n)$, hence by \eqref{e:FSpot}, $u_{j,p}\to u_j$ in $L^1_{loc}(U)$, 
as $p\to\infty$, for each $k+1\leq j\leq m$. Recall that by \eqref{e:FScomp}, 
$u_{j,p}\geq u_j-\frac{c''}{p}$ holds on $U$ for all $p$ sufficiently large and some constant $c''>0$. 
Using the  assumption that  $\Sigma(h_j)$ are in general position, 
\cite[Theorem 3.5, Corollary 3.6]{FS95} implies that 
$$v_k dd^cu_{1}\wedge\ldots\wedge dd^cu_{k-1} \wedge dd^cu_{k+1,p}\wedge\ldots\wedge dd^cu_{m,p}
\to v_k dd^cu_{1}\wedge\ldots\wedge dd^cu_{k-1} \wedge dd^cu_{k+1}\wedge\ldots\wedge dd^cu_{m}$$
weakly on $U$ as $p\to\infty,$ and that the right hand side has locally bounded mass on $U$.
Since $X$ is compact, we may cover $X$ by a finite number of open sets $U$ as above. 
Writing $\omega^{n-m+1}$ as a finite sum of smooth positive forms such 
that each form is supported in at least one open set $U$,
the lemma follows from the last limit.
\end{proof}

\medskip

Let us close the paper with more examples of H\"older metrics with singularities. 

\smallskip
\noindent
(1) Consider a projective manifold $X$ and a smooth divisor $\Sigma\subset X$ .
By \cite{Kob84,TY86}, if $L=K_X\otimes\cO_X(\Sigma)$ is ample, there exists a complete K\"ahler-Einstein metric
$\omega$ on $M:=X\setminus\Sigma$ with $\operatorname{Ric}_\omega=-\omega$.
This metric has Poincar\'e type singularities, described as follows. We denote by $\mathbb{D}$ the unit disc in $\C$.
Each $x\in\Sigma$ has a coordinate neighborhood $U_x$ such that
\begin{equation*} \label{compl12,17}
U_x\cong\mathbb{D}^n,\;\,x=0,\;\,U_x\cap\Sigma\cong\{z=(z_1,\ldots,z_n):\,z_1=0\},\;\,
U_x\cap M\cong\mathbb{D}^\star\times\mathbb{D}^{n-1}.
\end{equation*}
Then $\omega=\frac{i}{2}\sum_{j,k=1}^ng_{jk}dz_j\wedge d\overline z_k$ is quasi-isometric 
to the Poincar\'e type metric
\[
\omega_P=\frac{i}{2}\frac{dz_1\wedge d\overline{z}_1}{|z_1|^{2}(\log|z_1|^2)^{2}}+
\frac{i}{2}\sum_{j=2}^n dz_j\wedge d\overline{z}_j\,.
\]
Let $\sigma$ be the canonical section of $\cO_{X}(\Sigma)$ (cf.\ \cite[p.\,71]{MM07}) 
and denote by $h_\sigma$ the metric induced by $\sigma$ on $\cO_{X}(\Sigma)$ 
(cf.\ \cite[Example\,2.3.4]{MM07}). Note also that $c_1(\cO_X(\Sigma),h_\sigma)=[\Sigma]$ 
by \cite[(2.3.8)]{MM07}. Consider the metric
\begin{equation}\label{e:hsigma}
h_{M,\sigma}:=h^{K_M}\otimes h_\sigma\;\,\text{on}\;\,L\mid_{_M}=
K_M\otimes\cO_X(\Sigma)\mid_{_M}\cong K_M.
\end{equation}
Note that $L$ is trivial over $U_x$ and the metric $h_{M,\sigma}$ has a weight 
$\varphi$ on $U_x\cap M\cong\mathbb{D}^\star\times\mathbb{D}^{n-1}$ given by 
$e^{2\varphi}=|z_1|^2\det[g_{jk}]$. So $dd^c\varphi=-\frac{1}{2\pi}\ric_\omega>0$ and 
$\varphi$ is psh on $U_x\cap M$.
We see as in \cite[Lemma 6.8]{CM11} that $\varphi$
extends to a psh function on $U_x$, and $h_{M,\sigma}$ extends uniquely to a positively curved
metric $h^L$ on $L$. By construction, $h^L$ is a H\"older metric with singularities. 

\smallskip
\noindent
(2) Let us specialize the previous example to the case of Riemann surfaces.
Let $X$ be a compact Riemann surface of genus $g$ and let $\Sigma=\{p_1,\ldots,p_d\}\subset X$.
It follows from the Uniformization Theorem that the following conditions are equivalent:

(i) $U=X\setminus\Sigma$ admits a complete K\"ahler-Einstein metric $\omega$
with $\operatorname{Ric}_\omega=-\omega$,

(ii) $2g-2+d>0$,

(iii) $L=K_X\otimes\cO_X(\Sigma)$ is ample,

(iv) the universal cover of $U$ is the upper-half plane
$\mathbb{H}$.

\noindent
If one of these equivalent conditions is satisfied, the K\"ahler-Einstein metric 
$\omega$ is induced by the Poincar\'e metric on $\mathbb{H}$.
In local coordinate $z$ centered at $p\in\Sigma$ we have $\omega=\frac{i}{2}gdz\wedge d\overline{z}$,
where $g$ satisfies $c|z|^{-2}(\log|z|^2)^{-2}\leq g\leq c^{-1}|z|^{-2}(\log|z|^2)^{-2}$, for some $c>0$,
in a punctured neighborhood of $p$.
Note that  $\omega$ extends to a closed strictly positive $(1,1)$-current, 
i.\,e.\ a positive measure of finite mass, on $X$.
By \cite[Lemma 6.8]{CM11} there exists a singular metric $h^L$
on $L$  such that $c_1(L,h^L)=\frac{1}{2\pi}\omega$ on $X$. The weight of $h^L$
near a point $p\in\Sigma$ has the form $\varphi=\frac12\log(|z|^2g)$,
which is H\"older with singularities.

\smallskip
\noindent
(3) Let $X$ be a complex manifold, $(L,h^L_0)$ a holomorphic line bundle on $X$
with smooth Hermitian metric such that $c_1(L,h^L_0)$ is a K\"ahler metric.
Let $\Sigma$ be a compact divisor with normal crossings.
Let $\Sigma_1,\dots,\Sigma_N$ be the irreducible components of $\Sigma$, 
so $\Sigma_j$ is a smooth hypersurface in $X$. Let $\sigma_j$ be holomorphic 
sections of the associated holomorphic line bundle $\cO_{X}(\Sigma_j)$ 
vanishing to first order on $\Sigma_j$ and let $|\cdot|_j$ be a smooth Hermitian metric on 
$\cO_{X}(\Sigma_j)$ so that $|\sigma_j|_j<1$ and $|\sigma_j|_j=1/e$ outside a relatively compact 
open set containing $\Sigma$. Set 
\begin{equation*}\label{e:F}
\Theta_{\delta}=\Omega+\delta dd^cF,\text{ where }\delta>0\,,\:\:F=-\frac{1}{2}\sum_{j=1}^N\log(-\log|\sigma_j|_j).
\end{equation*}
For $\delta$ sufficiently small $\Theta_{\delta}$ defines the generalized Poincar\'e metric \cite[Lemma 6.2.1]{MM07},
\cite[Section 2.3]{CM11}.
For $\varepsilon>0$,
$$h_\varepsilon^{L}=h^{L}_0\prod_{j=1}^N(-\log|\sigma_j|_j)^\varepsilon$$
is a singular Hermitian metric on $L$ which is H\"older with singularities.
The curvature $c_1(L,h^L_\varepsilon)$ is a strictly positive current on $X$, provided that $\varepsilon$ 
is sufficiently small (cf.\ \cite[Lemma 6.2.1]{MM07}). 
When $X$ is compact the curvature current of $h_\varepsilon$ dominates 
a small multiple of $\Theta_\delta$ on $X\setminus\Sigma$.

\smallskip
\noindent
(4) 
Let $X$ be a Fano manifold. Fix a Hermitian metric $h_0$ on $K_X^{-1}$
such that $\omega:=c_1(K_X^{-1},h_0)$ is a K\"ahler metric.
We denote by $PSH(X,\omega)$ the set of $\omega$-plurisubharmonic functions on $X$. 
Let $\Sigma$ be a smooth divisor in the linear system defined by 
$K_X^{-\ell}$, so there
exists a section $s\in H^0(X,K_X^{-\ell})$ with
$\Sigma=\operatorname{Div}(s)$.

Fix a smooth metric $h$ on the bundle $\cO_X(\Sigma)$ and 
let $\beta\in[0,1)$.
A conic K\"ahler metric $\widehat\omega$ on $X$ with cone angle 
$\beta$ along $\Sigma$, cf.\ \cite{Do12,Ti96}, is a current 
$\widehat\omega=\omega_\varphi=\omega+dd^c\varphi\in c_1(X)$, 
where $\varphi=\psi+|s|_h^{2\beta}\in PSH(X,\omega)$ and
$\psi\in\Cc^\infty(X)\cap PSH(X,\omega)$. 
In a neighborhood of a point of $\Sigma$
where $\Sigma$ is given by $z_1=0$ the metric $\widehat\omega$
is equivalent
to the cone metric $\frac{i}{2}(|z_1|^{2\beta-2}dz_1\wedge 
d\overline{z}_1+
\sum_{j=2}^ndz_j\wedge d\overline{z}_j)$.

The metric $\widehat\omega$ defines a 
singular metric $h_{\widehat\omega}$ on $K_X^{-1}$ which is H\"older with singularities.
Its curvature current is 
$\ric_{\,\widehat\omega}:=c_1(K_X^{-1},h_{\widehat\omega})=(1-\ell\beta)
\widehat\omega+\beta[\Sigma]$, where
$[\Sigma]$ is the current of integration on $\Sigma$.

\end{document}